\documentclass[11pt]{article}

\usepackage{amssymb,amsfonts,amsthm,amsmath,mathrsfs}
\usepackage{graphicx}
\usepackage{wrapfig}
\usepackage{float}
\usepackage{pst-all}
\definecolor{darkblue}{rgb}{0,0,0.44} %22-84-113
\usepackage[pdfborder={0 0 0},pdfborderstyle={},colorlinks=true,linkcolor=darkblue,citecolor=black,urlcolor=blue]{hyperref}
\usepackage[ margin= 1 in]{geometry}

\newcommand{\R}{\mathbb{R}}

\newcommand{\supp}{\operatorname{supp}}

\newcommand{\Z}{\mathbb{Z}}

\newcommand{\A}{\mathcal{A}}

\newcommand{\N}{\mathbb{N}}

\renewcommand{\P}{\mathbb{P}}

\renewcommand\appendix{\par
\gdef\thetable{\Alph{table}}
\gdef\thefigure{\Alph{figure}}
\section*{Appendix: some properties of entropy}
\gdef\thesection{\Alph{section}}
\setcounter{section}{1}}

\theoremstyle{definition}
  \newtheorem{defin}{Definition}[section]
  \newtheorem*{setting}{Setting}
  \newtheorem{question}[defin]{Question}

\theoremstyle{plain}
  \newtheorem{thm}[defin]{Theorem}
  \newtheorem{main thm}{Theorem}
  \newtheorem{prop}[defin]{Proposition}
  \newtheorem{fact}{Fact}[section]
  \newtheorem{cor}[defin]{Corollary}
  \newtheorem{lemma}[defin]{Lemma}

\theoremstyle{remark}
  \newtheorem{remark}[defin]{Remark}
  
  \title{Subshifts with slow complexity and simple groups with the Liouville property}
  \author{Nicol\'as Matte Bon\thanks{Universit\'e Paris Sud; nicolas.matte.bon@ens.fr}}
\date{
May 2014}
\begin{document}
\maketitle
  \abstract{We study random walk on topological full groups of subshifts, and show the existence of  infinite, finitely generated, simple groups with the Liouville property. Results by Matui and Juschenko-Monod have shown that the derived subgroups of topological full groups of minimal subshifts provide the first examples of finitely generated, simple amenable  groups. We show that if the (not necessarily minimal) subshift has a complexity function that grows slowly enough (e.g. linearly), then every symmetric and finitely supported probability measure on the topological full group has trivial Poisson-Furstenberg boundary. We also get explicit upper bounds for the growth of F\o lner sets.}
\section{Introduction}

In the early 50s Graham Higman gave the first example of a finitely generated, infinite simple group \cite{Higman:simplefinitelygenerated}. Later,  Hall \cite{Hall},  Gorju\v{s}kin \cite{Gorjuskin}, and  Schupp \cite{Schupp} showed that any countable group can be embedded in a 2-generated simple group. Thus, finitely generated simple groups can be arbitrarily ``large'' in some sense. It is considerably less understood how ``small'' can such groups be, from the point of view of their  asymptotic geometry.
 
It follows from Gromov's theorem \cite{Gromov:polynomialgrowth} that a finitely generated simple group can not have polynomial growth, and it is an open question, due to Grigorchuk (see  \cite[Problem 15]{Grigorchuk:survey}), whether it can have sub-exponential growth. Recall that groups of sub-exponential growth are amenable. Recently,  Juschenko and  Monod  \cite{Juschenko-Monod:simpleamenable} have  proven that there do exist finitely generated, simple groups that are amenable; the groups that they consider were known to be simple and finitely generated  by results of Matui \cite{Matui:simple}.

 We consider here a third property of groups that lies between sub-exponential growth and amenability: the Liouville property for finite-range symmetric random walks. We prove that there exist  simple groups with the Liouville property.

A group equipped with a probability measure $(G,\mu)$ has \emph{the Liouville property} if  the Poisson-Furstenberg boundary  is trivial; equivalently, if every bounded $\mu$-harmonic function on $G$ is constant on the subgroup generated by the support of $\mu$. Here a function $f:G\to \R$ is said to be $\mu$-harmonic if $f*\mu=f$, where $(f*\mu)(g)=\sum_{h\in G}f(gh)\mu(h)$. If the support of $\mu$ generates $G$ the  measure is said to be \emph{non-degenerate}.

 When no measure is specified,  we say  that \emph{the group $G$ has the Liouville property} if $(G,\mu)$ has the Liouville property for  \emph{every symmetric}, \emph{finitely supported} probability measure $\mu$ on $G$, including degenerate measures. Finitely generated groups with sub-exponential growth have the Liouville property (this is due to Avez \cite{Avez}), and groups with the Liouville property are amenable. More precisely, a group is amenable if, and only if, it admits a symmetric non-degenerate measure $\mu$ with trivial Poisson-Furstenberg boundary  (one implication is due to Furstenberg, see \cite[Theorem 4.2]{Kaimanovich-Vershik}, the other to Kaimanovich and Vershik \cite[Theorem 4.4]{Kaimanovich-Vershik} and to Rosenblatt \cite{Rosenblatt}).
However, there are finitely generated amenable groups, such as the wreath product $\Z/2\Z\wr \Z^3$,  that  admit no \emph{finitely supported}, non-degenerate measures with trivial boundary, see Kaimanovich and Vershik \cite[Proposition 6.1]{Kaimanovich-Vershik}; on some amenable groups, a non-degenerate measure with trivial boundary might not even be chosen to have finite entropy by a result of Erschler \cite[Theorem 3.1]{Anna:Liouville}.   For a recent survey on Poisson-Furstenberg boundaries of random walks on discrete groups, see \cite{Anna:ICM}.

\begin{thm}\label{T:simple Liouville}
There exist  finitely generated infinite groups that are simple and have the Liouville property (for every symmetric, finitely supported probability measure). Moreover, there are uncountably many pairwise non-isomorphic such groups.
\end{thm}

Theorem \ref{T:simple Liouville} is a consequence of Theorem \ref{main} below.

The groups that we consider to prove Theorem \ref{T:simple Liouville} are a sub-class of the finitely generated simple groups discovered by Matui \cite[Theorem 4.9, Theorem 5.4]{Matui:simple}  and  considered by Juschenko and Monod \cite{Juschenko-Monod:simpleamenable}. They arise as the commutator subgroups of  the \emph{topological full groups} of some minimal subshifts, on which we assume that the  \emph{complexity} grows slowly enough (these notions are defined in Subsection \ref{S: topological full groups}). Our approach does not rely on results in \cite{Juschenko-Monod:simpleamenable} and yields a new proof of amenability of the groups that we consider.  It also shows amenability of the topological full groups of a class of non-minimal subshifts with slow complexity (see Subsection \ref{S: amenability}). 

\subsection{Cantor systems, subshifts, and topological full groups} \label{S: topological full groups}

Throughout the paper let $\Sigma$ denote a compact, metrizable, totally disconnected topological space, and let $\tau$ be a homeomorphism of $\Sigma$. 

The \emph{topological full group} of  the dynamical system $(\Sigma, \tau)$ is the group $[[\tau]]$ of homeomorphisms of $\Sigma$ that locally coincide with a power of $\tau$, namely the group of homeomorphisms $g\in\operatorname{Homeo}(\Sigma)$ such that for every $x\in \Sigma$ there exists an open neighbourhood $U$ of $x$ and an integer $k\in \Z$ for which $g|_U=\tau^k|_U$. In other words,  $g\in\operatorname{Homeo}(\Sigma)$ belongs to $[[\tau]]$ if and only if there exists a continuous function $k_g:\Sigma\to\Z$, the \emph{orbit cocycle}, such that
\[\forall x\in \Sigma, \qquad g(x)=\tau^{k_g(x)}(x).\] 
Recall that a topological dynamical system is said to be \emph{minimal} if every orbit is dense.

The  dynamical system $(\Sigma ,\tau)$  is called a \emph{Cantor system} whenever $\Sigma$ is homeomorphic to the Cantor set; in our setting, this amounts to non-existence of isolated points in $\Sigma$.
 
 Giordano, Putnam and Skau study in \cite{Giordano-Putnam-Skau:flipconjugacy} the topological full group of a Cantor minimal system. They show that the structure of this countable group characterizes completely the dynamics of such systems: two Cantor minimal systems $(\Sigma, \tau)$ and $(\Sigma', \tau')$ have isomorphic topological full groups if and only if $(\Sigma', \tau')$ is topologically conjugate to $(\Sigma, \tau)$ or to $(\Sigma,\tau^{-1})$ \cite[Corollary 4.4]{Giordano-Putnam-Skau:flipconjugacy}. The proof of this  uses results from Boyle's thesis \cite{Boyle} on topological orbit-equivalence, see  \cite[Theorem 3.2]{Boyle-Tomiyama}. In fact, also the structure of the \emph{commutator subgroup} of the topological full group of a Cantor minimal system  characterizes the system in the same way; this  is due to Bezuglyi and Medynets \cite[Theorem 5.2]{Bezuglyi-Medynets:flipconjugacy}.

Let $\mathcal{A}$ be a finite alphabet and endow $\A^\Z$ with the product of the discrete topology on $\A$. The \emph{shift} over $\A$ is the Cantor system $(\A^\Z, \tau)$ where  $\tau$ acts on $\A^\Z$ by
$$\tau:\: \cdots x_{-3}x_{-2}x_{-1}.x_{0}x_1x_2x_3\cdots\mapsto \cdots x_{-2}x_{-1}x_{0}.x_1x_2x_3x_4\cdots.$$
A \emph{subshift} is a dynamical system $(\Sigma, \tau)$ where $\Sigma\subset \A^\Z$ is a closed subset which is invariant under the shift. A subshift is not necessarily a Cantor system, note however that a \emph{minimal} infinite subshift is automatically a Cantor system.

Matui shows in \cite{Matui:simple} that if  $(\Sigma,\tau)$  is a minimal Cantor system, the \emph{commutator subgroup} $[[\tau]]'$ of the topological full group is simple (and infinite), see \cite[Theorem 4.9]{Matui:simple} and the remark above it; if moreover $(\Sigma, \tau)$ is a minimal subshift, $[[\tau]]'$ is also finitely generated \cite[Theorem 5.4]{Matui:simple}. In this case, he also proves in \cite{Matui:exponentialgrowth} that the group $[[\tau]]'$ contains free sub-semigroups and thus it has exponential growth.  Juschenko and  Monod \cite{Juschenko-Monod:simpleamenable} prove that for every minimal Cantor system $(\Sigma,\tau)$  the group $[[\tau]]$ is amenable; this was conjectured by  Grigorchuk and  Medynets \cite{Grigorchuk-Medynets:conjectureamenable}. See also Juschenko and  de la Salle \cite{Juschenko-Salle:wobbling} where a part of \cite{Juschenko-Monod:simpleamenable} is generalized and simplified using recurrence of the random walk on the orbital Schreier graphs of the group. This amenability result is now part of a more general amenability criterion due to Juschenko, Nekrashevych and de la Salle \cite{Juschenko-Nekrashevych-Salle:recurrentgrupoids}.

The \emph{complexity}, or \emph{word-complexity}, of a subshift $(\Sigma,\tau)$ is the function $\rho:\N\rightarrow \N$ that counts the number of words of length $n$ in the alphabet $\A$ that appear as sub-words of sequences in $\Sigma$.  See \cite{Complexity} for a survey on the complexity function. This function is closely related to topological entropy, as  it is well known that the topological entropy of $(\Sigma, \tau)$ is given by the sub-additive limit $h_{\operatorname{top}}(\tau)=\lim_{n\to \infty}\frac{1}{n}\log\rho(n)$. However we do not  need the notion of topological entropy here.

Our main result concerns another notion of entropy: the asymptotic entropy of random walks on $[[\tau]]$. Given a probability measure $\mu$ on a countable group $G$, its  \emph{entropy}  $H(\mu)$ is the quantity $H(\mu)=-\sum_{g\in G} \mu(g)\log\mu(g)\in[0,+\infty]$. The \emph{random walk entropy} (also called \emph{asymptotic} or \emph{Avez entropy}) is the limit
$$h(\mu)=\lim_{n\to\infty}\frac{1}{n}H(\mu^{*n}),$$
where $\mu^{*n}$ is the $n$-th convolution power of $\mu$. The limit exists by sub-additivity. The entropy criterion  of  Kaimanovich and Vershik \cite[Theorem 1.1]{Kaimanovich-Vershik} and Derrienic \cite{Derrienic}  states that if $H(\mu)<\infty$ (e.g.\ if $\mu$ is finitely supported) the Liouville property for $(G,\mu)$ is equivalent to $h(\mu)=0$.

\begin{thm}\label{main}
Let $(\Sigma, \tau)$ be a subshift without isolated periodic points. Suppose that  the complexity $\rho$ of $\Sigma$ satisfies 
\[\lim_{n\to\infty}(\frac{\log n}{n})^2\rho(n)=0.\] 
Then for every finitely supported symmetric probability measure $\mu$ on $[[\tau]]$ the random walk entropy vanishes.

 More precisely, there exists a constant $C>0$ such that for every $n\geq 1$
  \[H(\mu^{*n})\leq C\rho(\lceil C\sqrt{n\log n}\rceil )\log n.\]
\end{thm}

\begin{remark} \begin{enumerate}
\item Theorem \ref{main} applies in particular  if the complexity grows linearly, i.e.\ if there exist a constant $C>0$ such that  $\rho(n)\leq Cn$. More generally, it applies if there exist $\alpha\in [1,2)$ and $C>0$ such that $\rho(n)\leq Cn^\alpha$.  This includes some well-studied classes of minimal subshifts, see Section \ref{S: examples}. 
\item The measure $\mu$ in the statement is not assumed to be non-degenerate; in fact the whole topological full group is not necessarily  finitely generated (cf.\ \cite[Corollary 5.5]{Matui:simple}). 
\end{enumerate}
\end{remark}
\begin{remark}Theorem \ref{main}  implies Theorem  \ref{T:simple Liouville}: a subshift  $(\Sigma, \tau)$ satisfying  the assumption in the statement can be chosen to be minimal,  and in this case the commutator subgroup $[[\tau]]'$ is simple and finitely generated by the results of Matui \cite[Theorem 4.9, Theorem 5.4]{Matui:simple}. 
For example,  a class  of minimal subshifts satisfying the assumptions of Theorem \ref{main} is given by the \emph{Sturmian  subshifts} \cite{Morse-Hedlund} obtained by coding irrational rotations of the circle  (we recall this construction in Subsection \ref{S: irrational rotations}). Sturmian subshifts are minimal and have complexity $\rho(n)=n+1$, see \cite[Section 2.1.2]{Sturmian}. Moreover, a spectral argument shows that distinct rotations give rise to non-conjugate Sturmian subshifts.  By the results of Giordano, Putnam and Skau \cite[Corollary 4.4]{Giordano-Putnam-Skau:flipconjugacy}, and Bezuglyi and Medynets \cite[Theorem 5.2]{Bezuglyi-Medynets:flipconjugacy}, this provides uncountably many examples of finitely generated simple Liouville groups, as claimed in Theorem \ref{T:simple Liouville}.
\end{remark}
\subsection{Application to amenability and F\o lner function}\label{S: amenability}
Theorem \ref{main}  also implies that the topological full group $[[\tau]]$ is amenable, since amenability of a group is equivalent to amenability of its finitely generated subgroups. This shows that the topological full group is amenable even without the assumption that $(\Sigma, \tau)$ is  {minimal}, made in \cite{Juschenko-Monod:simpleamenable}, if we assume instead that the complexity of $\Sigma$ grows slowly enough.
 \begin{cor}\label{C: amenable}
 Under the assumptions of Theorem \ref{main}, the group $[[\tau]]$ is amenable.
 \end{cor}  

Note that some assumptions  on the subshift $(\Sigma, \tau)$ are needed to ensure that the group $[[\tau]]$ is amenable: a construction due to van Douwen \cite{Douwen} shows that there exist non-minimal subshifts $(\Sigma, \tau)$ such that $[[\tau]]$ contains non-abelian free subgroups, see also \cite[Proposition 3.7.1]{Cornulier:Bourbaki} for a more recent exposition. Van Douwen's construction yields subshifts with positive topological entropy. As far as we know, there is no known example of a Cantor system $(\Sigma, \tau)$ having zero topological entropy, and such that $[[\tau]]$ is non-amenable.

\begin{question}
What are the slowest possible growth rates for the complexity of a subshift $(\Sigma, \tau)$, having the property that $[[\tau]]$ contains non-abelian free subgroups (respectively  is non-amenable, respectively is non-Liouville)?
\end{question}

The upper bound for the entropy in Theorem \ref{main} also provides lower bounds for the return probabilities, and thus upper bounds for the growth of F\o lner sets of finitely generated subgroups of $[[\tau]]$; these estimates are new even when $(\Sigma,\tau)$ is minimal. Recall that the \emph{F\o lner function} of a finitely generated amenable group $G$ equipped with a finite symmetric generating set $S$ is the function $\operatorname{F\o l}_{G,S}:\mathbb{N}\to\mathbb{N}$ given by
\[\operatorname{F\o l}_{G,S}(n)=\min\{ |F|\:: F\subset G \:, \:|\partial_SF|\leq \frac{1}{n}|F|\},\]
where $\partial_S F=\{g\in F\::\:\exists s\in S,\:sg\notin F\}$. In the setting of minimal subshifts, de Cornulier \cite[Question 1]{Cornulier:Bourbaki} raises the question to to estimate the return probabilities and the  F\o lner function of finitely generated subgroups of $[[\tau]]$, and to determine if they depend on the choice of $(\Sigma, \tau)$. The next corollary is a step in this direction. 

\begin{cor}\label{C: return probability}\label{C: Folner}
\begin{enumerate}
\item Under the assumptions of Theorem \ref{main}, for every symmetric and finitely supported probability measure on $[[\tau]]$, there exists $C_1>0$ such that the return probabilities satisfy for every $n \geq 1$ 
\[\mu^{*2n}(e)\geq \frac{1}{C_1}\exp(-C_1\rho(\lceil C_1\sqrt{n\log n}\rceil)\log n).\]

\item Suppose moreover that there exists  $C>0$ and $\alpha\in [1,2)$ such that $\rho(n)\leq Cn^\alpha$. Then for any  finitely generated subgroup $G$ of $[[\tau]]$, every finite symmetric generating set $S$ of $G$ and for every $\varepsilon >0$, there exists  $C_2>0$ such that for every $n\geq 1$
\[\operatorname{F\o l}_{G,S}(n)\leq C_2\exp(C_2 n^{2\alpha/(2-\alpha)+\varepsilon}).\]
\end{enumerate}
\end{cor}
Details on the proof will be given in Section \ref{S: proof}, after the proof of Theorem \ref{main}

\begin{remark}
The stronger assumption on $\rho$ in  the second part of Corollary \ref{C: Folner} simplifies the statement but it is not essential: the lower bound on the return probabilities always implies an upper bound on the F\o lner function, see for instance \cite[Corollary 14.5]{Woess:book} or \cite{Pittet-Saloff:survey}. 
\end{remark}

In this context, we shall mention Bartholdi and  Vir\'ag's proof of amenability of the Basilica group \cite{Bartholdi-Virag:amenability}, in which the Liouville property was (implicitly)  used for the first time as a tool, to  prove amenability and  estimate return probabilities. Their ideas have been generalized by several authors \cite{Kaimanovich:munchaussen, Bartholdi-Kaimanovich-Nekrashevych, Brieussel:nonuniformgrowth, Amir-Angel-Virag:linear, Liouvilletrees}, who prove amenability and the Liouville property for several classes of groups acting on rooted trees. The groups that we consider here do not act on rooted trees and we need a different method.

\subsection{Outline of the proof and structure of the paper}
Let us give an outline of the proof of Theorem \ref{main}.
Given any subshift $(\Sigma,\tau)$, and  a finitely generated subgroup $G=\langle S\rangle$ of $[[\tau]]$, the orbital Schreier graph of any non-periodic point $x\in \Sigma$ admits a natural Lipschitz embedding into $\Z$. This embedding is given by $gx\mapsto k_g(x)$, where $k_g$ is the orbit cocycle. Thus, if $g_n$ is the left random walk on $G$,   $k_{g_n}(x)$ performs a random walk on a graph with vertex set $\Z$ and edges connecting integers with bounded difference. This fact,  using some general Gaussian estimates due to  Hebisch and Saloff-Coste \cite{Hebisch-Saloff},  implies that the maximal displacement up to time $n$ of the orbit cocycle $k_{g_n}(x)$ has typical size $\sqrt{n}$ and the tail of its distribution admits a Gaussian upper bound.  The key observation is to deduce from this  that  the cocycle  $k_{g_n}$ at time $n$ is with overwhelming probability  constant on a given {cylinder subset} of $\Sigma$, provided that the  number of letters defining the cylinder is big enough compared to  $\sqrt{n}$. Our assumption on the complexity of $\Sigma$ tells us that there are few cylinders, and implies that $g_n$ belongs with high probability to a finite subset $A_n$ of $G$ which has sub-exponentially growing cardinality. This is equivalent to having zero random walk entropy.\\

\emph{Structure of the paper.} Section 2 consists of the proof of a preliminary fact, Proposition  \ref{estimate max}, which is essentially an application of the Gaussian estimates in \cite{Hebisch-Saloff}. Section 3, which is the core of the paper, contains the proof of Theorem \ref{main}. Section 4 discusses examples of subshifts to which Theorem \ref{main} applies. Finally we recall in the Appendix some well-known properties of entropy.

\subsubsection*{Acknowledgements}
 The question whether  topological full groups provide examples of simple  Liouville groups was raised during  communications with K.\ Juschenko, V.\ Nekrashevych and M.\ de la Salle, to whom  I am also grateful  for several discussions about their recent paper \cite{Juschenko-Nekrashevych-Salle:recurrentgrupoids} and about  topological full groups. I would also like to  thank  K.\ Juschenko for careful reading and for useful comments;  V.\ Nekrashevych for bringing to my attention the group described in Subsection \ref{fibonacci}, which I found enlightening; and  M.\ de la Salle for inviting me at  UMPA in November 2013. 
 In a preliminary version, Theorem \ref{main} was stated assuming minimality of $(\Sigma, \tau)$ and this was not needed in the proof, I thank A. Erschler and K. Juschenko who suggested to eliminate this assumption.
 In addition, I am grateful to V. Berth\'e for pointing out several examples of application of Theorem \ref{main}; to Y.\ de Cornulier  for many useful remarks, to E. Fink and A. Stewart for reading a first version, to L. Saloff-Coste for pointing out \cite{Coulhon-Saloff:lineargrowth}, to R.\ Tessera for an interesting conversation on topological entropy. 
I am  especially grateful to my supervisor A. Erschler for many valuable discussions. Finally, I thank  an anonymous referee for several useful comments  and for suggesting Remark \ref{density non-periodic}. This work is partially  supported by the ERC starting grant GA 257110 ``RaWG''.

\section{Preliminaries}\label{S: preliminaries}

The aim of this section is to prove Proposition \ref{estimate max}, that we will use in next section to analyse the random walk on the Schreier graphs of the action of $G$ on $\Sigma$. These graphs  turn out to be one of the simplest kind of infinite graphs:  they have linear growth and admit a natural Lipschitz embedding into $\Z$. Random walks on graphs of linear growth are very well understood. Upper bounds for the transition probabilities can be deduced from a  general result due to  Hebisch and Saloff-Coste \cite{Hebisch-Saloff}, and  a matching lower bound  holds for graphs of linear growth as it is shown by Coulhon and Saloff-Coste in \cite{Coulhon-Saloff:lineargrowth}. The upper bound will be sufficient to our purpose, we recall it below.

Let $\Gamma$ be a graph of bounded degree. A Markov kernel $p(x,y)$ on the vertex set of $\Gamma$ is said to be \emph{nearest neighbour} if it is symmetric and $p(x,y)=0$ unless $x,y$ are neighbours in $\Gamma$. We make the standing assumption that $p$ is $\delta$-\emph{uniformly elliptic}, i.e. there  exists a uniform constant $\delta>0$ such that
\begin{equation}\label{p bounded below} \forall x,y\text{ which are neighbours in $\Gamma$ } \qquad p(x,y)\geq \delta.\end{equation}

Recall that if $\Gamma$ is infinite and connected and $p$ is $\delta$-uniformly elliptic, by  \cite[Corollary 14.6]{Woess:book} there exists a constant $C_1>0$ such that for every $n\geq 1$
$$\sup_{x,y}p_n(x,y)\leq C_1\frac{1}{\sqrt{n}}.$$
Moreover the constant $C_1$  above only depends on $\delta$. This last sentence follows easily by inspection of the proof of \cite[Corollary 14.6]{Woess:book} after observing that, with the notations defined at \cite[p.\ 39]{Woess:book}, every  non-empty finite set $A$ satisfies $a(\partial A)\geq \delta$, since $\Gamma$ is infinite and connected and thus  $\partial A$ contains at least one edge.

By \cite[Theorem 2.1]{Hebisch-Saloff} the above inequality can be improved to obtain the following.

\begin{prop}[Hebisch and Saloff-Coste, Corollary of Theorem 2.1 in \cite{Hebisch-Saloff}]\label{T: Hebisch-Saloff}
 Let $\Gamma$ be an infinite connected graph, and let $p$ be a symmetric nearest neighbour Markov kernel on $\Gamma$. Suppose also that $p$ is  $\delta$-uniformly elliptic for some $\delta>0$. 
Then there exist positive constants $C_1,D$ such that for every $n\geq 1$ and every $x,y$ vertices of $\Gamma$
\begin{equation}\label{Gaussian estimate}
p_n(x,y)\leq C_1\frac{1}{\sqrt{n}}\exp(-\frac{\operatorname{dist}_\Gamma(x,y)^2}{Dn}),\end{equation}
where $\operatorname{dist}_\Gamma$ is the graph distance on $\Gamma$. The constants $C_1$ and $D$ only depend on the uniform ellipticity constant $\delta$.

\end{prop}
\begin{remark}\label{R: Gaussian estimates}
Let $d$ be another distance on  $\Gamma$
and suppose that there is $K>0$ such that $d\leq K\operatorname{dist}_\Gamma$. Then the same estimate holds if $\operatorname{dist}_\Gamma$ is replaced by $d$, with possibly different constants $C_1, D$, where $C_1$ only depends on $\delta$ and $D$ depends on $\delta$ and on the Lipschitz constant $K$.
\end{remark}

Until the end of the section, we assume to be in the following setting. 

\begin{setting}\label{assumption section 2}
Let $K>0$ and let $\Gamma$ be a  graph which is $K$-Lipschitz embedded into $\Z$. In other words, the vertex set of $\Gamma$ is identified with a subset of $\Z$, and whenever $x,y\in \Z$ are the endpoints of an edge of $\Gamma$ we have $|x-y|\leq K$. Suppose also that 0 belongs to the vertex set of $\Gamma$. We shall consider two distances on $\Gamma$, the graph distance and the distance induced by  $\Z$.
Suppose that $\Gamma$ is endowed with  a symmetric nearest neighbour Markov kernel $p$ that is $\delta$-uniformly elliptic. Let $(X_n)_{n\in \N}$ be a Markov chain  with kernel $p$ started at 0. We wish to study 
\[\max_{j\leq n}|X_j|,\]
where $|\cdot|$ is the absolute value of $\Z$.
\end{setting}

The following Proposition will be used  in next section.
\begin{prop}\label{estimate max}
Let $\Gamma$ be a graph $K$-Lipschitz embedded into $\Z$,  endowed with a symmetric nearest neighbour Markov kernel $p$ which is $\delta$-uniformly elliptic for some $\delta>0$.  There exists positive real constants $C,D, a_0$ such that for every $a\geq a_0$ and every $n\geq 1$
\[\P(\max_{j\leq n}|X_j|\geq a\sqrt{n})\leq C \exp({-\frac{(a-a_0)^2}{D}}),\]
where $X_n$ is the Markov chain with kernel  $p$  started at 0, and $|\cdot|$ is the absolute value of $\Z$. The constants $C,D,a_0$ only depend on the Lipschitz constant $K$ and on the constant $\delta$.
\end{prop}
 We stress that the uniform control on the constants in Proposition \ref{estimate max} is crucial for the application that we need.

 The proof relies on Proposition \ref{T: Hebisch-Saloff} and on the following modification of the classical ``reflection principle'' for the random walk on $\Z$. The argument of the proof is standard.
\begin{lemma}\label{reflection principle}
With the same assumptions as in Proposition \ref{estimate max}, suppose moreover that $\Gamma$ is infinite. Then there exists a constant $b_0>0$, only depending on $K, \delta$, such that for every $x>0$ and every $n\in\N$
$$\mathbb{P}(\max_{j\leq n}|X_j|\geq x)\leq 2\mathbb{P}(|X_n|\geq x-
b_0\sqrt{n}).$$

\end{lemma}
\begin{proof} 

Since we assume that the graph is infinite, Proposition \ref{T: Hebisch-Saloff} applies. A straightforward computation using (\ref{Gaussian estimate}) with respect to the distance of $\Z$ (see Remark \ref{R: Gaussian estimates}) shows  that for every $n\geq 1$ and $b\geq 2$ (to ensure that $b-1/\sqrt{n}\geq 1$) we have
\begin{equation}\label{Gaussian tail}
\P(|X_n|\geq b\sqrt{n})\leq 2C_1\int_{b-\frac{1}{\sqrt{n}}}^\infty e^{-\frac{t^2}{D}}dt\leq C_2e^{-\frac{(b-1/\sqrt{n})^2}{D}}\leq C_2e^{-\frac{(b-1)^2}{D}}, \end{equation}
where $C_1, D$ are the constants from Proposition \ref{T: Hebisch-Saloff}, and $C_2=C_1D$.
For $y\in \Z$, write $\P_y$ for the law of $(X_n)_{n\in \N}$ started at $X_0=y$, while $\P$ denotes $\P_0$.
Obviously (\ref{Gaussian tail}) holds unchanged if $\P$ is replaced by $\P_y$ and $|X_n|$ by $|X_n-y|$. In particular there is  $b_0>0$, that only depends on $K,\delta$, such that for every $y\in\Z$ and every $n\in \N$ we have 
\begin{equation}\mathbb{P}_{y}(|X_n-y|>b_0\sqrt{n})\leq \frac{1}{2},\label{choice of b_0}\end{equation} this will be  $b_0$ in the  statement. 

 Set $S_n=\max_{j\leq n}|X_j|$. Consider the stopping time $T_x=\inf\{n\geq 0:\: |X_n|\geq x\}$, and observe that the event $\{S_n\geq x\}$ is equal to $\{T_x\leq n\}$. We have
\begin{align*}\P(S_n\geq x)=\P(T_x\leq n)&=\\
\P(T_x\leq n,\: |X_n|\geq x-b_0\sqrt{n})&+\P(T_x\leq n, \: |X_n|< x-b_0\sqrt{n})\leq\\
\P(|X_n|\geq x-b_0\sqrt{n})&+\P(T_x\leq n, \: |X_n-X_{T_x}|>b_0\sqrt{n}).\end{align*}
To bound the second summand, write
\begin{align*}\P(T_x\leq n,\:|X_n-X_{T_x}|> b_0\sqrt{n})=\sum_{j\leq n}\P(T_x=j, |X_{n}-X_j|> b_0\sqrt{n})=\\
\sum_{j\leq n}\sum_{y\in\Z}\P(T_x=j,\:X_j=y)\P_y(|X_{n-j}-y|>b_0\sqrt{n})\leq\\
 \frac{1}{2}\sum_{j\leq n}\sum_{y\in\Z}\P(T_x=j,\: X_j=y)=
\frac{1}{2}\P(T_x \leq n)=\frac{1}{2}\P(S_n\geq x),\end{align*}
where equality between the first and second line follows from Markov property,  and inequality between the second and third line holds since we chose  $b_0$ verifying (\ref{choice of b_0}). These two computations together imply that
\[\P(S_n\geq x)\leq \P(X_n\geq x-b_0\sqrt{n})+\frac{1}{2}\P(S_n\geq x),\]
which is a rephrasing of the desired inequality. \qedhere

\end{proof}

\begin{remark}
 Lemma \ref{reflection principle} also holds without the assumption that $\Gamma$ is infinite. Here is a way to see this: first prove Proposition \ref{estimate max}, that does not assume that the graph is infinite (in the proof, we will only need the current form of Lemma \ref{reflection principle}), then apply it to show that also in the finite case there exists  $b_0>0$, that only depends on $K$ and $\delta$, verifying (\ref{choice of b_0}); the rest of the proof holds with no change. 

\end{remark}

\begin{proof}[Proof of Proposition \ref{estimate max}]
Suppose at first that $\Gamma$ is infinite. Then Proposition \ref{estimate max} follows from Lemma \ref{reflection principle} (setting $x=a\sqrt{n})$ and from the inequality (\ref{Gaussian tail}) applied to $b=a-b_0$. The constants $C, D, a_0$ are $C=2C_2=2C_1D$, where $C_2$ is as in (\ref{Gaussian tail}), $C_1, D$ are the constants from Proposition \ref{T: Hebisch-Saloff}, and $a_0=b_0+1$, where $b_0$ is the constant from Lemma \ref{reflection principle}. All these constants only depend on the Lipschitz constant $K$ and on $\delta$.

The case of finite $\Gamma$  is readily reduced to the infinite case as follows. Fix $a$ and $n$. If the vertex set of $\Gamma$ is contained in the interval $[-\lceil a\sqrt{n}\rceil,+\lceil a\sqrt{n}\rceil]\subset \Z$ then $\P(S_n> a\sqrt{n})=0$ and the claim is correct. Otherwise, modify $\Gamma$ outside that interval to obtain an infinite graph $\tilde{\Gamma}$ with the same Lipschitz constant $K$, and a Markov kernel $\tilde{p}$ which is $\delta/2$-uniformly elliptic and coincides with $p$ on edges that entirely lie in that interval (this can clearly be achieved). Random walks on $(\Gamma, p)$ and $(\tilde{\Gamma}, \tilde{p})$ are naturally coupled until the exit time from $[-\lceil a\sqrt{n}\rceil,+\lceil a\sqrt{n}\rceil]$, in particular the distributions of the exit times for the two random walks are the same. The conclusion follows from the infinite case. \qedhere

\end{proof}
\section{Proof of Theorem \ref{main}}
\label{S: proof}
We now turn to the proofs of Theorem \ref{main} and of Corollary \ref{C: return probability}.\\

\emph{Throughout the section, we suppose that $(\Sigma, \tau)$ is a subshift without isolated periodic points, and that $\mu$ is a symmetric, finitely supported probability measure on $[[\tau]]$ (possibly degenerate).\\ 
We set $S=\supp\mu$ and $G=\langle S\rangle\le [[\tau]]$. }\\

For the moment we do not make any assumption on the complexity of $\Sigma$. The assumption on the complexity in Theorem \ref{main}  will  be used in the last part of the proof, as we will point out.

\begin{remark}\label{density non-periodic}
The set of non-periodic points is dense in $\Sigma$. To see this, observe that absence of isolated periodic points implies that, for any $n\in\N$, the finite set of $n$-periodic points has empty interior. By Baire's Theorem the set of all periodic points has empty interior, in other words non-periodic points are dense.
\end{remark}

\begin{defin}\label{orbit cocycle}
By  definition of topological full group, for every element $g\in [[\tau]]$ there exists a continuous,  locally constant function $k_g:\Sigma\to \Z$, called the \emph{orbit cocycle}, defined by the requirement
\[\forall x\in \Sigma, \qquad g(x)=\tau^{k_g(x)}(x).\]
  Note that $g$ is uniquely determined by $k_g$. Conversely, observe that the value of $k_g(x)$ is uniquely determined by $g$ if $x$ is a non-periodic point. By Remark \ref{density non-periodic}, non-periodic points are dense, so by continuity the function $k_g$ is uniquely determined by $g$ everywhere.  By compactness of $\Sigma$, $k_g$ takes finitely many values for every fixed $g\in [[\tau]]$ . We set
 \begin{equation} \label{Lipschitz constant} K=\max_{s\in S}\max_{x\in \Sigma} |k_s(x)|.\end{equation} \end{defin}
Note  that  the orbit cocycle verifies the cocycle rule
 \begin{equation}\label{cocycle} k_{gh}(x)=k_g(hx)+k_h(x).\end{equation}

  \begin{remark}\label{length bound}
 If $l_S$ is the word length on $G$ defined by $S$, the cocycle rule implies that  $|k_g(\cdot)|\leq Kl_S(g)$ point-wise.  
\end{remark}
 
 \begin{defin}
We fix the following notation. Given an integer $l>0$ and a finite word $w=w_{-l}\cdots w_{-1} w_0w_1\cdots w_l$ of length $2l+1$ in the alphabet $\A$, we denote by $\mathcal{C}_w$ the cylinder subset
 \[\mathcal{C}_w=\{x=\cdots x_{-1}.x_0x_1\cdots\in \Sigma\::\: x_i=w_i, \:\: \forall i=-l,\ldots l\}\subset \Sigma.\]
In what follows, the word \emph{cylinder} will always refer to a subset of $\Sigma$ of this form. The integer $l$ is called its \emph{depth}.
\end{defin}

We start with an elementary Lemma, which provides a criterion to ensure that $k_g$ is constant  cylinders. 
\begin{lemma}\label{coupling}

There exists $l_0\in\N$, which only depends on $(\Sigma,\tau)$ and on $S$, such that the following holds.

Let $n>0$ and $l>l_0$ be integers.  Let $h_1,\ldots, h_n\in S$ be any $n$-tuple of elements in the generating set,  and set $g_j=h_j\cdots h_1$ for every $j=1,\ldots, n$. Let $\mathcal{C}_w$ be a non-empty cylinder of depth $l$. Choose any non-periodic point  $x\in \mathcal{C}_w$ (which exists by Remark \ref{density non-periodic}) and suppose that 

\begin{equation}\label{max<l-l_0}\max_{j\leq n} |k_{g_j}(x)|\leq l-l_0.\end{equation}
 Then for every $j=1,\ldots,n$, the restriction of the orbit cocycle $k_{g_j}$ to $\mathcal{C}_w$ is constant. 
 \end{lemma}
\begin{proof}
The function $\Sigma\to\Z^S$ given by $x\mapsto (k_s(x))_{x\in S}$ is locally constant and takes finitely many values. Thus, the level sets of this function provide a finite partition $\mathcal{P}$ of $\Sigma$ into clopen sets such that for every generator $s\in S$ the restriction of  $k_s$ to every element of $\mathcal{P}$ is constant. 
After taking a refinement, we may suppose that $\mathcal{P}$ consists of cylinders. Let $l_0$ be the largest depth  of a cylinder in $\mathcal{P}$, this will be $l_0$ in the statement. In other words, if $x=\cdots x_{-1}.x_0x_1\cdots\in \Sigma$ and $s\in S$, in order to determine $k_s(x)$  it is sufficient to know the letters $x_{-l_0}, x_{-l_0+1},\ldots, x_0,\ldots,x_{l_0}$.

Now suppose that (\ref{max<l-l_0}) holds and let $y\in \mathcal{C}_w$. For clarity, suppose at first that $y$ is non-periodic. Let us prove by induction on $j\leq n$ that $k_{g_j}(x)=k_{g_j}(y)$. For $j=1$ observe that since $x,y$ are in a same cylinder of depth $l>l_0$, they lie in the same element of $\mathcal{P}$. Since $g_1=h_1\in S$, this implies that $k_{g_1}(x)=k_{g_1}(y)$.

 Suppose that the conclusion holds for $j$ and that $g_{j+1}=h_{j+1}g_j$ with $h_{j+1}\in S$. 

First note that if $x=\cdots x_{-1}.x_0x_1\cdots$ and $y=\cdots y_{-1}.y_0y_1\cdots$ are in $\mathcal{C}_w$, by the inductive hypothesis we have
\begin{align*}
g_jx=\tau^{k_{g_j}(x)}(x)=&\cdots x_{k_{g_j}(x)-1}.x_{k_{g_j(x)}}x_{k_{g_j}(x)+1}\cdots,\\
g_jy=\tau^{k_{g_j}(y)}(y)=\tau^{k_{g_j}(x)}(y)=&\cdots y_{k_{g_j}(x)-1}.y_{k_{g_j(x)}}y_{k_{g_j}(x)+1}\cdots.
\end{align*}
Since $x$ and $y$ agree on letters at distance at most $l$ from the letter at position 0 and by  the assumption (\ref{max<l-l_0}) we have $|k_{g_j}(x)|\leq l-l_0$,  we conclude that the sequences $g_jx$ and $g_jy$ agree on letters at distance at most $l_0$ from the letter at position 0. Hence they lie in the same element of $\mathcal{P}$. Since $h_{j+1}\in S$, it follows that $k_{h_{j+1}}(g_jx)=k_{h_{j+1}}(g_jy)$.
Thus, using the cocycle rule (\ref{cocycle}) and again the inductive hypothesis
$$k_{g_{j+1}}(x)=k_{h_{j+1}}(g_jx)+k_{g_j}(x)=k_{h_{j+1}}(g_jy)+k_{g_j}(y)=k_{g_{j+1}}(y),$$
which completes the induction. We have proven that for every $j\leq n$ the restriction of $k_{g_j}$ to the set of non-periodic points in  $\mathcal{C}_w$ is constant. Since this set is dense by Remark \ref{density non-periodic}, this implies that the restriction of $k_{g_j}$ to $\mathcal{C}_w$ is constant.
\qedhere

\end{proof}

\emph{From now on let $(g_n)_{n\in\N}$ be the left random walk on $G$ driven by $\mu$, i.e. $g_n=h_n\cdots h_1$ where $(h_i)_{i\geq 1}$ is a sequence of independent $G$-valued random variables, with distribution $\mu$.}\\

We first look at the process $(k_{g_n}(x))_{n\in\N}$ for a fixed $x\in \Sigma$. 

\begin{remark}
Let $x\in \Sigma$ be a non-periodic point, and let $O(x)=\{\tau^j(x)\}_{j\in\Z}$ be the $\tau$-orbit of $x$. Since $x$ is not periodic, $O(x)$ can be identified with $\Z$ via the map
\begin{align*}
\iota_x:O(x)&\to \Z\\
\tau^j(x)&\mapsto j.\end{align*}
Let $Gx\subset O(x)$ be the $G$-orbit of $x$ and let $\Gamma_x$ be the corresponding \emph{Schreier graph} with respect to the generating set $S$, that is, the undirected graph with vertex set $Gx$ and where $y,z\in Gx$ are connected by an edge if there exists $s\in S$ such that $y=sz$. It is straightforward to check that the restriction of $\iota_x$ to  $\Gamma_x$ is $K$-Lipschitz for the constant $K$ defined in (\ref{Lipschitz constant}). Identify the vertex set of $\Gamma_x$ with a subset of $\Z$ using the map $\iota_x$.
This identification  sends $gx\in Gx$ to $k_g(x)\in\Z$. It follows that the process $(k_{g_n}(x))_{n\in\N}$ is the position in $\Z$ of   $(g_nx)_{n\in \N}$, which is in turn a nearest neighbour random walk on the graph $\Gamma_x$, with Markov kernel $p(y,z)=\sum_{s:sy=z}\mu(s)$. 

Finally, note that the Markov kernel  $p$ on $\Gamma_x$  is $\delta$-uniformly elliptic (see (\ref{p bounded below})) with  $\delta=\min_{s\in S}\mu(s)$.  Note also that both constants $K$ and $\delta$ are independent from the choice of $x$. We are in  position to apply Proposition \ref{estimate max} to $\Gamma=\Gamma_x$ and  $X_n=k_{g_n}(x)$, and we have obtained the following Lemma.
\end{remark}

\begin{lemma}\label{max k}
There exist positive constants $C_1,D, a_0$ such that for every $a\geq a_0$ every $n\geq 1$ and every non-periodic $x\in \Sigma$ we have
$$\P(\max_{j\leq n} |k_{g_j}(x)|\geq a\sqrt{n})\leq C_1 \exp({-\frac{(a-a_0)^2}{D}}).$$
 The constants $C_1,D,a_0$ are independent from the choice of $x$.
\end{lemma}

Combining Lemma \ref{coupling} with Lemma \ref{max k} we get:

\begin{cor}\label{single cylinder}
In the situation of Lemma \ref{max k}, for every large enough $n\in \mathbb{N}$ and every $L\geq 1$ the following holds. If $\mathcal{C}_w$ is a cylinder of depth $\lceil \sqrt{L n\log n}\rceil$, we have
\[\mathbb{P}(k_{g_n} \text{\emph{ is not constant on }} \mathcal{C}_w)\leq C_1 n^{-L/4D},\]
where $C_1$ and $D$ are the  constants from Lemma \ref{max k}.
\end{cor}
\begin{proof}
Pick  $x\in \mathcal{C}_w$ non-periodic. By Lemma \ref{coupling} if $k_{g_n}$ is  not constant on $\mathcal{C}_w$ we have $\max_{j\leq n}|k_{g_j}(x)|>\lceil \sqrt{L n\log n}\rceil-l_0$. To bound the probability of this event we apply Lemma \ref{max k} with $a=\sqrt{L \log n}-l_0/\sqrt{n}$. To simplify the formula, we take $n$ large enough so that $\sqrt{L \log n}-l_0/\sqrt{n}-a_0\geq \frac{1}{2}\sqrt{L \log n}$;  note that since we assume that $L\geq 1$, the minimal  $n$ verifying this can be chosen to depend only on  $a_0$ and $l_0$, and not on $L$ nor on the cylinder under consideration. With this choice, Lemma \ref{max k} gives the bound claimed in the Corollary.
\end{proof}

We are now ready to prove Theorem \ref{main}. \\

\emph{From now, we fully assume to be  in the situation of Theorem \ref{main}. In particular we assume that the complexity $\rho$ satisfies the assumption in the statement. We keep all the other notations introduced above in this section.}\\

To prove Theorem \ref{main}, we exhibit a sequence of finite subsets $A_n\subset G$ that grow sub-exponentially and such that $\mu^{*n}(A_n)\to 1$. Existence of such sets  is equivalent to the vanishing of  the random walk entropy by Fact \ref{entropy vanishing criterion}. We will then deduce entropy estimates using Fact \ref{quantitative entropy vanishing criterion}.

 \begin{defin}[Definition of the sets $A_n$]\label{definition of An}
 Suppose to be in the situation of Theorem \ref{main}. Choose $L>8D$, where $D$ is the constant from Corollary \ref{single cylinder}.
 
  Let  $\tilde{A}_n\subset G$ to be the set of all elements $g\in G$ such that for every cylinder $\mathcal{C}_w$ with depth  $\lceil \sqrt{L n\log n}\rceil$ the restriction of $k_g$  to $\mathcal{C}_w$ is constant. Finally set
\[A_n= \tilde{A}_n\cap B_{G,S}(n),\] 
 where $B_{G,S}(n)$ is the ball of radius $n$ in the word metric induced by $S$.

Denote by $\mathscr{C}(n)$ the collection of all non-empty cylinders with depth $\lceil \sqrt{L n\log n}\rceil$.  Every cylinder in $\mathscr{C}(n)$ is determined by a word of length $2\lceil \sqrt{L n\log n}\rceil+1$.
To simplify the notations, we extend the function $\rho$  to a piecewise affine function defined on $\R_+$, still denoted $\rho$, and we introduce a constant $C_2>0$ so that for every $n\geq 2$
\begin{equation}\label{cardinality cn}
|\mathscr{C}(n)|= \rho(2\lceil \sqrt{L n\log n}\rceil+1)\leq \rho(C_2\sqrt{n\log n}).
\end{equation}
 \end{defin}

\begin{lemma} \label{An subexponential}
 The cardinality of $A_n$ grows sub-exponentially, i.e.\ $\frac{1}{n}\log|A_n|\to 0$.
\end{lemma}
  \begin{proof}

 By construction, the word length of any $g\in A_n$ does not exceed $n$. Hence by Remark \ref{length bound} for any $g\in A_n$ we have $|k_{g}(\cdot)|\leq Kn$ point-wise.

An element $g\in A_n$ is uniquely determined by the value of $k_g$ on every cylinder in $\mathscr{C}(n)$, and  we observed that $k_g$ does not exceed $Kn$ in absolute value, thus using (\ref{cardinality cn}):
 \begin{equation}\label{size of A} |A_n|\leq |\{-Kn,\cdots,0,\cdots,Kn\}|^{|\mathscr{C}(n)|}\leq(2Kn+1)^{\rho(C_2\sqrt{n\log n})}.\end{equation}
 The assumption on $\rho$ in the statement of Theorem \ref{main} guarantees that $\frac{1}{n}\log|A_n|\to0$. \qedhere
 \end{proof}
 
\begin{lemma} \label{An likely} We have $\mu^{*n}(A_n) \to 1$. \end{lemma}
\begin{proof}
 We have  $\mu^{*n}(A_n)=\mathbb{P}(g_n\in A_n)$, where $(g_n)_{n\in \N}$ is the  random walk. Obviously $g_n$ lies in the ball of radius $n$, thus  we only need to prove that $k_{g_n}$ is constant on every cylinder in $\mathscr{C}(n)$  with probability tending to one as $n\to \infty$.
  
 We have
\begin{align*}1-\mu^{*n}(A_n)=\mathbb{P}(\exists \mathcal{C}_w\in \mathscr{C}(n)\::\: k_{g_n}\text{ is not constant on }\mathcal{C}_w)\leq\\
\sum_{\mathcal{C}_w\in \mathscr{C}(n)}\mathbb{P}( k_{g_n}\text{ is not constant on }\mathcal{C}_w).\end{align*}

   To bound to the last sum, use Corollary \ref{single cylinder}, which applies to  the cylinders in $\mathscr{C}(n)$ by the choice made in Definition \ref{definition of An}. Using (\ref{cardinality cn}) we get
     \begin{equation}\label{bound of q}1-\mu^{*n}(A_n)\leq C_1n^{-L/4D}|\mathscr{C}(n)|\leq C_1n^{-L/4D}\rho(C_2 \sqrt{n\log n}), 
 \end{equation} 
 where $C_1$ is the constant from Corollary \ref{single cylinder}. Note that by the choice of $L$ in Definition \ref{definition of An} we have $L/4D>2$, and this implies that $\mu^{*n}(A_n)\to 1$ by the assumption on $\rho$ (that guarantees that $ \frac{1}{n^\beta}\rho(C_2\sqrt{ n\log n})\to 0$ for every $\beta\geq 1$).
   \qedhere
\end{proof}

\begin{proof}[End of the proof of Theorem \ref{main}]
The fact that the asymptotic entropy vanishes immediately follows from Lemma \ref{An subexponential} and Lemma \ref{An likely} using Fact \ref{entropy vanishing criterion}.
To see that the claimed upper bound for $H(\mu^{*n})$ holds, use Point 1 from Fact \ref{quantitative entropy vanishing criterion} together with (\ref{size of A}) and (\ref{bound of q}) to get that there exists a constant $C_3>0$ such that
 \begin{align*} 
 H(\mu^{*n})\leq \log|A_n| + n(1-\mu^{*n}(A_n))\log|S|+\log 2 \leq \\
 C_3  \rho(C_2\sqrt{n\log n})\log n +C_3+C_3n^{1-L/4D}\rho(C_2\sqrt{n\log n}).
 \end{align*}
Finally observe that the last summand  tends to zero by the choice of $L$ in Definition \ref{definition of An} and by the assumption on $\rho$. This implies the claimed upper bound for the entropy.
\end{proof}

\begin{proof}[Proof of Corollary \ref{C: return probability}]
Using Point 2 of Fact \ref{quantitative entropy vanishing criterion} together with (\ref{size of A}) we get that there exists $C_4>0$ so that for every $n\geq 1$
\[\mu^{*2n}(e)\geq \frac{1}{C_4}\exp(-C_4\log n \rho(C_2\sqrt{n\log n})).\]
The claim on the F\o lner function follows from Nash inequality theory, see for instance \cite[Corollary 14.5 (b)]{Woess:book} or \cite{Pittet-Saloff:survey}.
\end{proof}

\section{Examples}
\label{S: examples}

\subsection{Irrational rotations}
\label{S: irrational rotations}

One of the first historical examples of infinite minimal subshifts are the so-called \emph{Sturmian subshifts} associated to irrational rotations of the circle, first defined in \cite{Morse-Hedlund}. For material on Sturmian subshifts see \cite[Chapter 2]{Sturmian}.

Let $\alpha\in(0,1)$ be irrational. Consider the irrational rotation
$R_\alpha:x\mapsto x+\alpha$ on the unit circle $\R/\Z$.  Take as alphabet $\A=\{a,b\}$, and let $\phi:\R/\Z\to\{a,b\}$ be defined by $\phi(x)=a$ if $x\in [0,\alpha)\: \operatorname{mod }1$ and $\phi(x)=b$ otherwise.
Let $\Sigma_\alpha\subset\{a,b\}^\Z$ be the closure of $\{(\phi(R_\alpha^j(x)))_{j\in \Z}\: :\: x\in \R/\Z\}$.
The subshift $( \Sigma_\alpha, \tau_\alpha)$ is minimal \cite{Hedlund} (in particular, it has no periodic points), and  its complexity is given by $\rho(n)=n+1$  see for instance \cite[Theorem 2.1.13]{Sturmian} (this is the slowest possible complexity for an infinite subshift). 

It follows  from Theorem \ref{main} that every finitely generated subgroup of $[[\tau_\alpha]]$ has the Liouville property. In particular, $[[\tau_\alpha]]'$ is an infinite,  finitely generated, simple Liouville group.

 If $\beta\neq \alpha$ are in $[0, 1/2)$, it is well known that the topological conjugacy classes of $(\Sigma_\alpha, \tau_\alpha)$ and $(\Sigma_\beta, \tau_\beta)$ are distinguished by their spectrum, see for instance \cite[pp. 3-4]{Cornulier:Bourbaki}. Thus the two systems are not topologically conjugated.

By  results of Giordano, Putnam and Skau \cite[Corollary 4.4]{Giordano-Putnam-Skau:flipconjugacy} and  Bezuglyi and  Medynets \cite[Theorem 5.2]{Bezuglyi-Medynets:flipconjugacy}, the commutator subgroups of the topological full groups of two Cantor minimal systems $(\Sigma, \tau)$ and $(\Sigma', \tau')$ are isomorphic if and only if $(\Sigma',\tau')$ is topologically conjugate to $(\Sigma,\tau)$ or to $(\Sigma, \tau^{-1})$.

It follows that when $\alpha$ runs in $[0,1/2)$, the groups  $[[\tau_\alpha]]'$  provide uncountably many pairwise non-isomorphic examples of simple Liouville groups, as claimed in Theorem \ref{T:simple Liouville}.

Let $S$ be a finite symmetric generating set of $[[\tau_\alpha]]'$, and let $\mu$ be a symmetric probability measure supported on $S$. Using the explicit value of the complexity, Corollary \ref{C: return probability} gives that for every $\varepsilon>0$ there exists a constant $C$ such that 
\begin{align*}
\mu^{*2n}(e)\geq \frac{1}{C}\exp(-Cn^{1/2+\varepsilon});\\
\operatorname{F\o l}_{[[\tau_\alpha]]', S}(n)\leq C\exp(Cn^{2+\varepsilon}).
\end{align*}
\subsection{Substitutions}\label{fibonacci}\label{S: substitutions}
Another source of subshifts with slow complexity are the \emph{substitution dynamical systems}. See the books \cite{Fogg, Queffelec} for a survey. These provide examples of both minimal and non-minimal subshifts satisfying the assumptions in Theorem \ref{main}.

Let $\A^*$ be the set of finite words in the alphabet $\A$. A \emph{substitution} is a map $\psi:\A\to\A^*$. Such a map obviously extends to $\A^*$ by concatenation, the extension is still denoted $\psi:\A^*\to\A^*$. It makes thus sense to consider iterations of $\psi$. We shall make two standing assumptions:
\begin{enumerate}
\item there exists a letter $a\in \A$ such that $\psi(a)$ begins with an $a$;
\item for every $b\in\A$ the length of $\psi^n(b)$ tends to infinity as $n\to\infty$. 

\end{enumerate}

Any substitution $\psi$ satisfying Conditions 1 and 2 defines a subshift $(\Sigma_\psi, \tau_\psi)$ by the following construction.

Condition 1 above implies that $\psi^n(a)$ is a prefix of $\psi^{n+1}(a)$ for every $n\in \N$. Thus we can pass to the limit and obtain a right-infinite sequence $\psi^{\infty}(a)$.  Chose an arbitrary letter $b\in\A$ and consider the bi-infinite sequence ${x}=\cdots bbbb\psi^\infty(a)\in\A^\Z$. Define $\Sigma_\psi$ as the set of cluster points of $(\tau^n({x}))_{n\geq 0}$, where $\tau$ is the shift. This set does not depend on the choice of $b\in \A$, it is a closed non-empty subset of $\A^\Z$ and it is invariant under the shift. We denote $(\Sigma_\psi, \tau_\psi)$ the subshift obtained in this way.

 A finite sub-word of $\psi^\infty(a)$ appears as a sub-word of a sequence in $\Sigma_\psi$ if and only if it appears infinitely many times in $\psi^\infty(a)$. 
In fact,  $\Sigma_\psi\subset \A^\Z$  coincides with the subset of $\A^\Z$ consisting of words such that every finite sub-word appears as a sub-word of $\psi^\infty(a)$, but the above construction gives more precise information.

We recall a partial case of a result due to Pansiot (see \cite[Theorem 4.7.1]{Complexity}).
\begin{thm}[cf. Theorem 4.7.55 in \cite{Complexity}]
If $\psi$ satisfies Conditions 1 and 2 above, there exists a constant $C>0$ such that the complexity $\rho$ of $\Sigma_\psi$ satisfies for every $n\geq 2$
\[\rho(n)\leq Cn\log n.\]

\end{thm}
\begin{cor}\label{C: substitutions}
Under the same assumptions, every finitely generated subgroup of $[[\tau_\psi]]$ has the Liouville property.  In particular $[[\tau_\psi]]$ is amenable.
\end{cor}
\begin{proof}[Proof of Corollary \ref{C: substitutions}]
We may assume that $\psi^\infty(a)$ is not eventually periodic, since in this case $\Sigma_\psi$ is finite and $[[\tau_\psi]]$ is a finite group. To apply Theorem \ref{main}, we only need to check that no periodic point is isolated in $\Sigma_\psi$. Let  $x=(x_j)_{j\in\Z}\in\Sigma_\psi$ be a $n$-periodic point, and let $w=x_0\cdots x_{n-1}$ be its period. Fix $l>0$ and consider the word $w^{2l}$ consisting of $2l$ concatenations of $w$. Since $w^{2l}$ appears in $x$, it appears infinitely often in $\psi^\infty(a)$. Since $\psi^\infty(a)$ is not eventually periodic, we deduce that it admits  an infinite sequence of sub-words of the form  $(w^{2l}w_k)_{k\geq 1}$ where $w_k\neq w$ has length $n$. Since there are only finitely many possibilities for $w_k$, there exists $w'\neq w$ of length $n$ such that  $w^{2l}w'$ appears infinitely often. It follows that $w^{2l}w'$ appears as a sub-word of a sequence in $\Sigma_\psi$. By shift-invariance, there exists $y^{(l)}=(y^{(l)}_j)_{j\in\Z}\in\Sigma_\psi$ such that $(y^{(l)}_j)_{-nl\leq j\leq nl-1}=w^{2l}=(x_j)_{-nl\leq j\leq nl-1}$  and $(y^{(l)}_j)_{nl\leq j\leq nl+n-1}=w'\neq w= (x_j)_{nl\leq j\leq nl+n-1}$. We have $y^{(l)}\neq x$ and $y^{(l)}\to x$ as $l\to\infty$, thus $x$ is not isolated.\qedhere

\end{proof}

A substitution is said to be \emph{primitive} if there exists $l>0$ such that for every ordered pair of letters $x,y\in\A$ the letter $x$ appears in $\psi^l(y)$. Note that whenever $\psi$ is a primitive substitution and the alphabet contains at least two letters, Condition 2 above is automatically verified, and Condition 1  is always verified up to passing to an iteration of $\psi$. We recall in the next Proposition two well-known facts about primitive substitutions.
\begin{prop}[see Proposition 5.5 and Proposition 5.12 in \cite{Queffelec}]\label{P: primitive substitutions}
Let $\psi$ be a substitution satisfying Conditions 1 and 2 above. Then $\psi$ is primitive if and only if  $(\Sigma_\psi, \tau_{\psi})$ is minimal. Moreover in this case there exists $C>0$ such that for every $n\geq 1$
\[\rho(n)\leq Cn,\]
where $\rho$ is the complexity of $\Sigma_\psi$.
\end{prop}
It follows that for any primitive substitution $\psi$, the group $[[\tau_\psi]]'$ is a finitely generated simple Liouville group. Moreover Corollary \ref{C: return probability} combined with Proposition \ref{P: primitive substitutions} provides estimates for its return probabilities and its F\o lner function.

 Primitive substitutions subshifts, as well as the Sturmian subshifts  described in Subsection \ref{S: irrational rotations}, belong to the class of minimal subshfits with \emph{linearly growing} complexity. This is a well studied class of minimal subshifts, see for instance Ferenczi  \cite[Proposition 5]{Ferenczi:linearcomplexity}.

Let us recall an explicit  finitely generated group which appears as an example  in  \cite{Juschenko-Nekrashevych-Salle:recurrentgrupoids}.
Consider the \emph{Fibonacci substitution} on the alphabet $\A=\{a,b\}$ given by 
\begin{align*}
\psi\::\:&a\mapsto ab\\
&b\mapsto a.
\end{align*}
We obtain the right-infinite sequence
$$\psi^\infty(a)=abaababaabaababaababa\cdots.$$
The Fibonacci substitution is primitive, thus it yields a minimal subshift $(\Sigma_\psi, \tau_\psi)$. In fact, the subshift $(\Sigma_\psi, \tau_\psi)$ is also a Sturmian subshift, obtained from the irrational rotation by the golden ratio $\phi=(1+\sqrt{5})/2$, see \cite[Example 2.1.1]{Sturmian} (the reader may consult \cite[Section 2.1, Proposition 5.2.21]{Fogg} for an example of a primitive substitution subshift which is not conjugate to a Sturmian subshift, given by the \emph{Thue-Morse} substitution).  Let $\alpha,\beta,\gamma$ act on a sequence $x=\cdots x_{-1}.x_0x_1\cdots\in \Sigma_\psi$ by

\begin{align*}&\left\{\begin{array}{lr}
\alpha(x)=\tau_\psi(x) & \text{ if }x_{-1}x_0=aa\\
\alpha(x)=\tau_\psi^{-1}(x) &\text{ if } x_{-2}x_{-1}=aa\\
\alpha(x)=x& \text{ otherwise,}
\end{array}\right.\left\{\begin{array}{lr}
\beta(x)=\tau_\psi(x) & \text{ if }x_{-1}x_0=ba\\
\beta(x)=\tau_\psi^{-1}(x) &\text{ if } x_{-2}x_{-1}=ba\\
\beta(x)=x& \text{ otherwise,}
\end{array}\right.\\ &\left\{\begin{array}{lr}
\gamma(x)=\tau_\psi(x) & \text{ if }x_0=b\\
\gamma(x)=\tau_\psi^{-1}(x) &\text{ if } x_{-1}=b\\
\gamma(x)=x& \text{ otherwise.}
\end{array}\right.\end{align*}
One can check that $\alpha,\beta,\gamma$ are elements of $[[\tau_\psi]]$ and are involutions.  Set $G=\langle \alpha,\beta,\gamma\rangle$. By Corollary \ref{C: substitutions},  the group $G$ is Liouville. Let $\mu$ be the measure equidistributed on $S=\{\alpha,\beta,\gamma\}$. For every point $x\in\Sigma_\psi$ the Schreier graph of $G$ acting on the  orbit of $x$ with respect to the generating set $S$ is isomorphic to $\Z$, with additional loops based at every vertex. If $(g_n)_{n\in\N}$ is the random walk on $G$ then $(k_{g_n}(x))_{n\in\N}$ performs a lazy random walk on $\Z$ that at each step chooses  whether to  stay, go left, or go right, each with probability $\frac{1}{3}$. 

We conclude this section by giving an example of a non-minimal subshift, to which Theorem \ref{main} applies. Consider the non-primitive substitution $\psi$ on the alphabet $\A=\{a,b\}$ given by
\begin{align*}
\psi\::\:&a\mapsto aba\\
&b\mapsto bb.
\end{align*}
This substitution verifies Conditions 1 and 2 stated at the beginning of this subsection. Since $\psi$ is not primitive, the subshift $(\Sigma_\psi, \tau_\psi)$ is not minimal (in fact it contains the constant sequence $\cdots bbb\cdots$). The complexity $\rho$ of $\Sigma_\psi$ is computed explicitly in \cite[Paragraph 4.10.5]{Complexity}, where it is shown that the limit $\lim_{n\to\infty}\rho(n)/n\log\log n$ exists.  It follows from Corollary \ref{C: substitutions} that $[[\tau_\psi]]$ is amenable and its finitely generated subgroups are Liouville.

\subsection{Toeplitz subshifts}\label{S: Toeplitz}
 \emph{Toeplitz subshifts} provide examples of applications of Theorem \ref{main} to minimal subshifts with super-linear complexity. The complexity of Toeplitz subshifts is studied by Cassaigne and Karhum\"aki in \cite{Toeplitz}. I am grateful to Val\'erie Berth\'e for bringing this example to my attention.

Let $\mathcal{A}$ be a finite alphabet, and consider the alphabet $\A\cup\{*\}$, where $*\notin\A$ is an extra letter, thought of as a ``hole''.
Let $w$ be a finite word in the alphabet $\A\cup\{*\}$ starting with a letter in $\A$. To such a word, we associate a right infinite \emph{Toeplitz word} in the alphabet $\A$, denoted $T_\infty(w)$ and obtained as follows. Let 
$T_1(w)=w^\N=wwww\cdots$  be the right-infinite periodic word in the alphabet $\A\cup\{*\}$ with period $w$. At step $i$, build a new  right-infinite word $T_i(w)$ in the alphabet $\A\cup\{*\}$ obtained from $w^\N$ by ``filling the holes''  with the sequence $T_{i-1}(w)$, i.e.\  the first  appearance of $*$ in $w^\N$ is replaced by the first letter of $T_{i-1}(w)$, the second appearance by the second letter, and so on. In the limit, $T_i(w)$ converges to a right infinite word $T_\infty(w)$ in the alphabet $\mathcal{A}$ (we use here the assumption that the first letter of $w$ belongs to $\A$).
For example the word $w=a*ab*a$ yields the Toeplitz word 
\[T_\infty(a*ab*a)=aaabaaaaabbaaaabaaaaabaa\cdots.\]
We obtain a \emph{Toeplitz subshift} $(\Sigma_w, \tau_w)$, where $\Sigma_w\subset\A^\Z$ consists of all bi-infinite sequences so that any finite sub-word appears in $T_\infty(w)$. Any Toeplitz subshift is minimal, see \cite[p. 499]{Toeplitz}, and it is  infinite as soon as the Toeplitz word $T_\infty(w)$ is not eventually periodic. We recall in the next theorem  a particular case of the results from \cite{Toeplitz}.
\begin{thm}[Cassaigne and Karhum\"aki, see Theorems 4 and 5 in \cite{Toeplitz}] 
 Suppose that $w$ has length $p>1$ and $q\geq 1$ appearances of $*$, and that $\operatorname{gcd}(p,q)=1$. Assume also that $w$ contains at least two different letters in $\A$. Then the subshift $(\Sigma_w, \tau_w)$ is infinite and minimal. Moreover there exist positive constants $C_1\leq C_2$ such that complexity $\rho$ of $\Sigma_w$ satisfies 
\[{C_1}n^{{\log p}/{\log (p/q)}}\leq \rho(n)\leq C_2n^{{\log p}/{\log (p/q)}}\]
for every $n\geq 1$.
\end{thm}
This produces examples of minimal subshifts with super-linear complexity satisfying the assumptions in Theorem \ref{main} 
\begin{cor}
Suppose moreover that $q^2<p$. Then every finitely generated subgroup of $[[\tau_w]]$ has the Liouville property. In particular $[[\tau_w]]'$ is a finitely generated, infinite, simple Liouville group.
\end{cor}

\appendix
We list here some well-known properties of  entropy, and in particular Facts \ref{entropy vanishing criterion}  and \ref{quantitative entropy vanishing criterion} that we have used to establish Theorem \ref{main}. Basic definitions and notations on random walk entropy have been introduced in Subsection \ref{S: topological full groups}, before the statement of Theorem \ref{main}.

The first fact is  elementary and we omit the proof.
\begin{fact}[Elementary properties]\label{elementary properties of entropy}
Let $X$ be a countable set and $\nu$ be a probability measure on $X$. 
\begin{enumerate}
\item If the support of $\nu$ is finite, we have $H(\nu)\leq \log |\operatorname{supp}(\nu)|$.
\item Suppose that $\nu=\sum_{i\geq 0}\alpha_i\nu_i$, where the $\nu_i$ are probability measures on $X$ and the $\alpha_i$ are positive reals such that $\sum_{i\geq 0}\alpha_i=1$. Then
\[H(\nu)\leq \sum_{i\geq 0}\alpha_iH(\nu_i)-\sum_{i\geq 0} \alpha_i\log\alpha_i.\]
\end{enumerate}
\end{fact}

The next fact  was proven independently by Kaimanovich and Vershik \cite[Theorem 2.1]{Kaimanovich-Vershik}, and by Derrienic \cite{Derrienic}.

\begin{fact}[``Shannon's Theorem'', \cite{Kaimanovich-Vershik, Derrienic}]\label{Shannon}
Let $\mu$ be a probability measure on a countable group $G$ such that $H(\mu)<\infty$, and let $h(\mu)$ be the random walk entropy. Then for almost every path $(g_n)_{n\in \N}$ of the random walk with step measure $\mu$ we have
\[h(\mu)=\lim_{n\to\infty}-\frac{1}{n}\log\mu^{*n}(g_n).\]
\end{fact}

We now come to the criterion of vanishing of the random walk entropy that has been used to establish Theorem \ref{main}. We say that a positive function $f:\N\to \R$ \emph{grows sub-exponentially} if $\frac{1}{n}\log f(n)\to 0$.
\begin{fact}\label{entropy vanishing criterion}
Let $\mu$ be a probability measure on a countable group $G$ with finite entropy $H(\mu)$. The following are equivalent:
\begin{itemize}
\item[(i)] $h(\mu)=0$;
\item[(ii)] there exists a sequence of finite subsets $A_n\subset G$ with sub-exponentially growing  cardinality and such that $\mu^{*n}(A_n)\to 1$;
\item[(iii)] there exists a sequence of finite subsets $A_n\subset G$ with sub-exponentially growing cardinality  and such that $\mu^{*n}(A_n)$ is uniformly bounded away from zero.
\end{itemize}
\end{fact}
This is a well-known reformulation of the entropy criterion of Kaimanovich and Vershik \cite[Theorem 1.1]{Kaimanovich-Vershik} and Derrienic \cite{Derrienic}. For the convenience of the reader we provide a proof.
\begin{proof}
(i) $\Rightarrow$ (ii). Fix $\varepsilon >0$.  Set
\[B_n^\varepsilon=\{g\in G\: :\: \mu^{*n}(g)\geq e^{-\varepsilon n}\}.\]
If $g\notin  B_n^\varepsilon$, we have $-\frac{1}{n}\log\mu^{*n}(g)\geq \varepsilon$, thus
\[\frac{1}{n}H(\mu^{*n})=-\frac{1}{n}\sum_{g\in G}\mu^{*n}(g)\log\mu^{*n}(g)\geq-\frac{1}{n}\sum_{g\notin B^n_\varepsilon}\mu^{*n}(g)\log\mu^{*n}(g)\geq(1-\mu^{*n}(B_n^\varepsilon))\varepsilon.\]
Since by {(i)} the left hand side tends to 0, this proves that $\mu^{*n}(B_n^{\varepsilon})\to 1$. On the other hand the choice of $B_n^\varepsilon$  implies that $|B_n^\varepsilon| \leq e^{\varepsilon n}$ and thus
\[\limsup_{n\to\infty}\frac{1}{n}\log |B_n^\varepsilon|\leq \varepsilon.\] A diagonal extraction argument provides a sequence $\varepsilon_n $ decreasing to zero such that $A_n=B_n^{\varepsilon_n}$ verifies (ii).

(ii)$\Rightarrow$ (iii) is obvious.

(iii)$\Rightarrow$ (i). Let $A_n$ be as in (iii). Fix $\varepsilon>0$ and set
\[\tilde{A}_n=\{g\in A_n\::\:\mu^{*n}(g)\leq e^{-\varepsilon n}\}.\]
Observe that since $\tilde{A}_n\subset A_n$ and $|A_n|$ grows sub-exponentially we have
\[\mu^{*n}(\tilde{A}_n)\leq e^{-\varepsilon n}|\tilde{A}_n|\leq e^{-\varepsilon n}|{A}_n|\to 0.\]
Hence if we set $A'_n=A_n\setminus \tilde{A_n}$ we have that $\mu^{*n}(A'_n)$ is still bounded away from zero. Moreover every $g\in A'_n$ verifies $\mu^{*n}(g)\geq e^{-\varepsilon n}$. Since $\mu^{*n}(A'_n)$ is bounded away from zero, the set of random walk paths $(g_n)_{n\in \N}$ with the property that $g_n\in A'_n$ for infinitely many $n$ has positive probability. For such a path we have $-\frac{1}{n}\log\mu^{*n}(g_n)\leq \varepsilon$ infinitely many times. The random walk entropy can be computed using Fact  \ref{Shannon} (Shannon's Theorem) and restricting  to the (positive-measured) set of paths with this property. We conclude that $h(\mu)\leq \varepsilon$. This implies that $h(\mu)=0$ since $\varepsilon$ was arbitrary.
\end{proof}
\begin{fact}[Quantitative version of Fact \ref{entropy vanishing criterion}]\label{quantitative entropy vanishing criterion}
Let $\mu$ be a probability measure with finite entropy on a countable group $G$ such that $h(\mu)=0.$
\begin{enumerate}
\item If $\mu$ is finitely supported with support $S$ and $A_n$ are finite sets as in  (ii) of Fact \ref{entropy vanishing criterion}, the entropy of the convolutions $H(\mu^{*n})$ satisfies
\[H(\mu^{*n})\leq \log |A_n| +n(1-\mu^{*n}(A_n))\log |S|+\log 2.\]
\item If $\mu$ is symmetric and $A_n$ are as in  Fact \ref{entropy vanishing criterion} (ii) or (iii), there exists $C>0$ such that the return probabilities at even times satisfy for every $n\geq 1$
\[\mu^{*2n}(e)\geq \frac{1}{C|A_{2n}|}.\]
\end{enumerate}
\end{fact}

\begin{proof}
1. Let $\nu_{1,n},\nu_{2,n}$ be the restrictions of $\mu^{*n}$ to $A_n$  and its complement, respectively. Set $p_n=\mu^{*n}(A_n)$ and $q_n=1-p_n$, so that $\mu^{*n}=p_n\nu_{1,n}+q_n\nu_{2,n}$. Observe that $\supp\nu_{1,n}\subset A_n$ and $\supp\nu_{2,n}\subset S^n$. Using Point 2 and  Point 1 of Fact \ref{elementary properties of entropy} we have
  \begin{align*} 
 H(\mu^{*n})\leq &p_nH(\nu_{1,n})+q_nH(\nu_{2,n})-p_n\log p_n-q_n\log q_n\leq\\
 &H(\nu_{1,n})+q_nH(\nu_{2,n})+\log2\leq \log|A_n| + nq_n\log|S|+\log 2.
 \end{align*}
 where we have bounded above $p_n$ by 1 and we have used that $-p_n\log p_n-q_n\log q_n\leq \log 2$.
 
 2. It is well-known that if $\mu$ is symmetric the return probability at even times $\mu^{2n}(e)$ maximize $\mu^{*2n}(g)$ for $g\in G$, namely by Cauchy-Schwartz
 \begin{align*}\mu^{*2n}(g)=\sum_{h\in G} \mu^{*n}(gh^{-1})\mu^{*n}(h)\leq \sqrt{\sum_{h\in G} \mu^{*n}(gh^{-1})^2\sum_{h\in G} \mu^{*n}(h)^2}=\\\sum_{h\in G}\mu^{*n}(h)^2=\sum_{h\in G} \mu^{*n}(h^{-1})\mu^{*n}(h)=\mu^{*2n}(e),\end{align*}
where equality between the first and the second line is a variable change in one of the sums, and we have used symmetry in the second line.
Let $C>0$ be such that $1/C$ is a lower bound for $\mu^{*2n}(A_{2n})$. We have
\[\frac{1}{C}\leq \mu^{*2n}(A_{2n})=\sum_{g\in A_{2n}}\mu^{2n}(g)\leq |A_{2n}| \mu^{*2n}(e),\]
which completes the proof.
\end{proof}

\bibliography{these}
\bibliographystyle{alpha}

\end{document}